\documentclass[12pt,reqno]{amsart}
\usepackage{amsmath, amssymb, amsthm, mathrsfs}
\usepackage{stmaryrd} %\llbracket, \rrbracket
\usepackage{color} %, pxfonts}
\usepackage[all]{xy}
\usepackage{yhmath}
\usepackage{soul}
\usepackage{bbm} % \mathbbm

\usepackage{hyperref}
\usepackage{leftidx} %\ltrans, \leftidx
\usepackage{mathabx} %\widecheck
\usepackage{mathtools} %\prescript
\usepackage{rotating} %\rotatebox
\usepackage{stmaryrd} %\llbracket, \rrbracket
\hypersetup{colorlinks,linkcolor=blue,urlcolor=cyan,citecolor=blue}

\usepackage{enumitem }

\theoremstyle{definition}
\newtheorem{Def}{Definition} %[section] %[subsection]

\newtheorem{example}[Def]{Example}

\newtheorem{rem}[Def]{Remark}

\theoremstyle{plain}

\newtheorem{thm}[Def]{Theorem}

\newtheorem{cor}[Def]{Corollary}
\newtheorem{conj}[Def]{Conjecture}

\newtheorem{theorem}[Def]{Theorem}
\newtheorem{lemma}[Def]{Lemma}

\newtheorem{proposition}[Def]{Proposition}

\newcommand{\g}{\mathfrak{g}}
\newcommand{\JJ}{\mathcal{A}}
\newcommand{\la}{\lambda}
\newcommand{\hwv}{v_\la^+}

\title[Odd singular vector Formula]{Odd singular vector formula for general linear Lie superalgebras}
\author[Jie Liu, Li Luo and Weiqiang Wang]{Jie Liu, Li Luo and Weiqiang Wang}
\address{Department of mathematics, Shanghai Key Laboratory of Pure Mathematics and Mathematical Practice, East China Normal University, Shanghai 200241, China}
    \email{jie$_{-}$liu@math.uni-kiel.de, lluo@math.ecnu.edu.cn, ww9c@virginia.edu}
\address{ Department of Mathematics\\ University of Virginia\\ Charlottesville, VA 22904}

\begin{document}

\begin{abstract}
We establish a closed formula for a singular vector of weight $\lambda-\beta$ in the Verma module of highest weight $\lambda$ for Lie superalgebra $\mathfrak{gl}(m|n)$ when $\lambda$ is atypical with respect to an odd positive root $\beta$. It is further shown that this vector is unique up to a scalar multiple, and it descends to a singular vector, again unique up to a scalar multiple, in the corresponding Kac module when both $\lambda$ and $\lambda-\beta$ are dominant integral.
\end{abstract}

\maketitle

%%%%%%%%
%%%%%%%%

\section{Introduction}

%According to a classical result of Verma, there exists a unique (up to a scalar multiple) homomorphism between Verma modules $M(s_\alpha \la)$ and $M(\la)$ of a semisimple Lie algebra $\g$.

For a general basic Lie superalgebra $\g$ (which is a suitable super generalization of semisimple Lie algebras) with a non-degenerate invariant bilinear form $(\cdot, \cdot)$, the linkage principle is not completely controlled by the Weyl group; cf. the book \cite{CW12}. Besides the more familiar linkage for weights in a same Weyl group orbit, a so-called $\beta$-atypical weight $\la$ can be linked to $\la -\beta$ for a positive odd isotropic root $\beta$; here by $\beta$-atypical we mean $(\la+\rho, \beta)=0$. This is one of the fundamental differences between representation theories of Lie superalgebras and Lie algebras.

Now let $\g$ be the general linear Lie superalgebra $\mathfrak{gl}(m|n)$ over the complex number field $\mathbb{C}$. The main result of this note is a simple closed formula for an (odd) singular vector $S_{-\beta} \hwv$ of weight $\la-\beta$ in the Verma module $M(\la)$ of highest weight $\la$, when $\la$ is $\beta$-atypical. We then show that such a singular vector is unique up to a scalar multiple; see Theorem~\ref{thm:Vsingular}. In other words, we have
\begin{equation}  \label{eq:1d}
\mathrm{Hom}_\mathfrak{g} \big(M(\lambda-\beta),M(\lambda) \big)=\mathbb{C}.
\end{equation}
This can be regarded as a super analogue of a classical theorem of Verma for semisimple Lie algebras. Serganova \cite{Se96} showed earlier the Hom space in \eqref{eq:1d} is non-vanishing by some indirect argument. We note that formulae for various singular vectors associated to an {\em even} reflection in a Verma module of a basic Lie superalgebra were recently established in \cite{Sa17}; cf. also \cite{CW17}.

We readily convert our odd singular vector formula to a simple closed formula for the corresponding odd Shapovalov element; see Corollary~\ref{cor:Selement}. Musson has studied systematically Shapovalov elements in the setting of basic Lie superalgebras, and in particular, he obtained in \cite[Theorems~9.1, 9.2]{Mu17} very different and more involved expressions for odd Shapovalov elements in terms of non-commutative determinants.

For a dominant integral weight $\la$ for $\mathfrak{gl}(m|n)$, the Kac module $K(\la)$ is by definition a maximal finite-dimensional quotient of the Verma module $M(\la)$. Assume $\la$ is $\beta$-atypical and in addition $\la-\beta$ is dominant integral. We show that the aformentioned vector $S_{-\beta} \hwv$ descends to a singular vector in $K(\la)$, and such a singular vector in $K(\la)$ is unique up to a scalar multiple; see Theorem~\ref{thm:Ksingular}.
% (all we need to verify here is that $S_{-\beta} \hwv\neq 0$ in the quotient module).
Indeed this uniqueness is equivalent to an earlier simple observation by Serganova that
$\dim \mathrm{Hom}_\mathfrak{g} \big(K(\lambda-\beta),K(\lambda) \big) \leq 1.$
Hence we have provided a new constructive proof that $\mathrm{Hom}_\mathfrak{g} \big(K(\lambda-\beta),K(\lambda) \big) =\mathbb C$, an old result of Serganova \cite[Theorem 5.5]{Se96}. This result played a fundamental role in the earlier approaches \cite{Se96, Br03} toward Kazhdan-Lusztig theory for the category of {\em finite-dimensional} $\mathfrak{gl}(m|n)$-modules.

The identity \eqref{eq:1d} was also known in \cite{CW17} for the exceptional Lie superalgebra $D(2,1;\zeta)$ where explicit odd singular vector formulas were established. We conjecture that the identity \eqref{eq:1d} holds for basic Lie superalgebas in general. We show that essentially the same odd singular vector formula for $\mathfrak{gl}(m|n)$ is valid for half of the odd isotropic roots of $\mathfrak{osp}$ Lie superalgebras.
It seems highly nontrivial and very interesting to generalize the odd singular vector formula of this paper to the other half of odd roots for $\mathfrak{osp}$ and the other basic Lie superalgebras.

\subsection*{Acknowledgement}
The authors thank Bin Shu for raising the question about odd singular vectors. LL is partially supported by the Science and Technology Commission of Shanghai Municipality (grant No. 18dz2271000) and the NSF of China (grant No. 11671108, 11871214). WW is partially supported by the NSF grant DMS-1702254, and he thanks ECNU for the hospitality and support during his visit.

%%%%%%%%
%%%%%%%%
\section{Formula for odd singular vectors}

%%%
\subsection{The preliminaries}

%%%\subsection{General linear Lie superalgebras}

Let
\begin{equation*}
I=I_{m|n}=\{\overline{m},\overline{m-1},\ldots,\overline{1},1,2,\ldots, n\}
\end{equation*}
be a set with a total order
\[
\overline{m}<\overline{m-1}<\cdots<\overline{1}<1<2<\cdots<n.
\]

Let
$\mathfrak{g}=\mathfrak{gl}(m|n)=\mathfrak{g}_{\overline{0}}\oplus \mathfrak{g}_{\overline{1}}$
be the general linear Lie superalgebra with standard basis $\{E_{i,j}~|~i,j\in I\}$. %, where the even (resp. odd) part is
%$$\mathfrak{g}_{\overline{0}}=\left\{
%\begin{pmatrix}A_{m\times m}&0_{m\times n}\\0_{n\times m}&D_{n\times n}
%\end{pmatrix}
%\right\}\quad\mbox{(resp.
%$\mathfrak{g}_{\overline{1}}=\left\{
%\begin{pmatrix}0_{m\times m}&B_{m\times n}\\C_{n\times m}&0_{n\times n}
%\end{pmatrix}
%\right\}$)}.$$
Then $\mathfrak g$ admits a natural $\mathbb{Z}$-grading
$\mathfrak{g}=\mathfrak{g}_{-1}\oplus\mathfrak{g}_{0}\oplus\mathfrak{g}_1$ where $\mathfrak{g}_{0}=\mathfrak{g}_{\overline{0}}$ and $\mathfrak{g}_{\pm 1}$ is of the form
\[
\mathfrak{g}_{-1}=\left\{
\begin{pmatrix}0 &0 \\C_{n\times m}&0
\end{pmatrix}
\right\},\qquad
\mathfrak{g}_{1}=\left\{
\begin{pmatrix}0 &B_{m\times n}\\0 &0
\end{pmatrix}
\right\}.
\]
Let $\mathfrak{h}$ be the Cartan subalgebra consisting of diagonal matrices and let $\mathfrak{g}=\mathfrak{n}^-\oplus \mathfrak{h}\oplus \mathfrak{n}^+$ be a triangular decomposition,
where $\mathfrak{n}^-$ (resp. $\mathfrak{n}^+$) consists of lower (resp. upper) triangular matrices. Let $\{\delta_m,\ldots,\delta_1,\epsilon_1,\ldots,\epsilon_n\}\subset \mathfrak{h}^*$ be the dual basis of
$\{E_{\overline{m},\overline{m}},\ldots,E_{\overline{1},\overline{1}},E_{1,1},\ldots,E_{n,n}\}\subset\mathfrak{h}$.
A nondegenerate symmetric form $(\cdot, \cdot)$ on $\mathfrak{h}^*$ induced by an invariant symmetric form on $\mathfrak{g}$ is given by
\begin{equation*}
(\delta_i,\epsilon_k)=0 \quad \forall i,k,
\qquad
(\delta_i,\delta_j)= -(\epsilon_i,\epsilon_j) =
\begin{cases}
1 & \text{ if } i=j,\\
0 & \text{ if } i\neq j.
\end{cases}
%\delta_{ij}.
\end{equation*}

%%%\subsection{Root system}

The set of simple roots and the set of roots are
\begin{align*}
\Pi &=\{\delta_m-\delta_{m-1},\ldots,\delta_2-\delta_1,\delta_1-\epsilon_1,\epsilon_1-\epsilon_2,\ldots,\epsilon_{n-1}-\epsilon_n\},
\\
\Phi &=\{\delta_i-\delta_j,\epsilon_p-\epsilon_q,\pm(\delta_i-\epsilon_p)~|~\overline{m}\leq \overline{i}\neq \overline{j}\leq \overline{1},1\leq p\neq q\leq n\}.
\end{align*}

Furthermore, the sets of positive even roots and positive odd roots are given by
\begin{align*}
\Phi_0^+ &=\{\delta_i-\delta_j,\epsilon_p-\epsilon_q~|~\overline{m}\leq \overline{i}<\overline{j}\leq \overline{1},1\leq p<q\leq n\},
\\
\Phi_1^+ &=\{\delta_i-\epsilon_p~|~\overline{m}\leq \overline{i}< 1\leq p\leq n\}.
\end{align*}
The set of even roots and the set of odd roots are
$\Phi_0=\Phi_0^+\cup(-\Phi_0^+)$ and $\Phi_1=\Phi_1^+\cup (-\Phi_1^+)$, respectively.
The root space decomposition is $\mathfrak{g}=\mathfrak{h}\oplus (\bigoplus_{\alpha\in\Phi}\mathfrak{g}_{\alpha})$, where
\[
\mathfrak{g}_{\delta_i-\delta_j} =\mathbb{C}E_{\overline{i},
\overline{j}}, \quad \mathfrak{g}_{\epsilon_p-\epsilon_q} =\mathbb{C}E_{p,q},\quad
\mathfrak{g}_{\delta_i-\epsilon_p} =\mathbb{C}E_{\overline{i},p},\quad
\mathfrak{g}_{\epsilon_p-\delta_i} =\mathbb{C}E_{p,\overline{i}}.
\]

%%%\subsection{Weight}
A weight $\lambda\in\mathfrak{h}^*$ is {\em integral} if $(\lambda,\alpha)\in\mathbb{Z}$ for all $\alpha\in\Phi_0$, and is {\em dominant} if $(\lambda,\alpha)\in\mathbb{Z}_{\geq0}$ for all $\alpha\in\Phi_0^+$. We denote by $P^+$ the set of all integral dominant weights.
We recall the Weyl vector
\begin{equation*}
\rho=\frac{1}{2}\Big(\sum_{\alpha\in\Phi_0^+}\alpha-\sum_{\alpha\in\Phi_1^+}\alpha\Big)
=\sum_{i=1}^mi\delta_i-\sum_{j=1}^nj\epsilon_j-\frac{m+n+1}{2}{\mathbf 1}_{m|n},
\end{equation*}
where ${\mathbf 1}_{m|n}=\sum_{i=1}^m\delta_i-\sum_{j=1}^n\epsilon_j$.
If there is a positive odd root $\beta\in\Phi_1^+$ such that $(\lambda+\rho,\beta)=0$, then the weight $\lambda$ is called {\em atypical} (or more precisely, {\em $\beta$-atypical}).

In this paper, a weight $\lambda=a_m\delta_m+\cdots+a_1\delta_1+b_1\epsilon_1+\cdots+b_n\epsilon_n\in\mathfrak{h}^*$
will be denoted by
\[
\lambda=(a_m,\ldots,a_1~|~b_1,\ldots,b_n).
\]
The following clearly holds:
\begin{align}
\lambda ~\mbox{is integral}~ &\Leftrightarrow a_i-a_j, b_p-b_q\in\mathbb{Z}, \forall \overline{m}\leq \overline{i}<\overline{j} \leq\overline{1}, 1\leq p<q\leq n;\nonumber\\
\lambda ~\mbox{is dominant}~ &\Leftrightarrow a_i-a_j, b_p-b_q\geq 0, \forall \overline{m}\leq \overline{i}<\overline{j} \leq\overline{1}, 1\leq p<q\leq n;\nonumber\\
\lambda ~\mbox{is $(\delta_s-\epsilon_t)$-atypical}~&\Leftrightarrow a_s+b_t+s-t=0.\label{asbt}
\end{align}

%%%
\subsection{Singular vectors in Verma modules}

A non-zero vector $v$ in a $\mathfrak{g}$-module $V$ is called {\em singular} if $\mathfrak{n}^+\cdot v=0$.
%It is clear that $v\in V$ is singular if and only if it is annihilated by the simple root vectors.
%, i.e.,
%$$E_{\overline{i},\overline{i-1}}\cdot v=E_{\overline{1},1}\cdot v= E_{j,j+1}\cdot v=0, \quad (\forall 1< i\leq m, 1\leq j< n).$$
Recall the Verma module of highest weight $\lambda$ is
$M(\lambda)=\mathrm{Ind}_{\mathfrak{h}\oplus\mathfrak{n}^+}^{\mathfrak{g}}\mathbb{C}1_\lambda,
%\cong \mathrm{U}(\mathfrak{g})\otimes_{\mathrm{U}(\mathfrak{h}\oplus\mathfrak{n}^+)}\mathbb{C}1_\lambda,
$
where $h\cdot 1_\lambda=\lambda(h)1_\lambda$ and $\mathfrak{n}^+\cdot 1_\lambda=0$;  the highest weight vector of $M(\la)$ is denoted by $\hwv =1\otimes 1_\la$.

Let $\beta=\delta_s-\epsilon_t$. Assume $\lambda=(a_m,a_{m-1},\ldots,a_1~|~b_1,b_2,\ldots,b_n)$ is a $\beta$-atypical weight.
Define a sequence of scalars
\begin{align}
\label{cvalue1}
c_{\overline{s-1}}=a_s-a_{s-1},
\quad
 c_{\overline{s-2}}=a_s-a_{s-2}+1,
\quad\quad \ldots,
\quad
c_{\overline{1}}=a_s-a_1+s-2;
\\
\label{cvalue2}
c_{t-1}=b_t-b_{t-1}-1,
\quad
 c_{t-2}=b_{t}-b_{t-2}-2,
\quad
 \ldots,
\quad
c_1=b_t-b_1-t+1.
\end{align}
Denote
\begin{equation} \label{eq:JJ}
\JJ =\big\{ J=\{j_1,j_2,\ldots, j_p\}\subset I \mid \overline{s}<j_1<j_2<\cdots<j_p<t, \; 0\le p \leq s+t-2 \big\},
\end{equation}
where it is understood that $J=\emptyset$ if $p=0$.
Then
\[
\{E_J:=E_{t,j_p}E_{j_p,j_{p-1}}\cdots E_{j_1,\overline{s}}~|~J=\{j_1,j_2,\ldots, j_p\}\in\JJ\}
\]
forms a basis of the $(-\beta)$-weight subspace of $\mathrm{U}(\mathfrak{n}^-)$; here it is understood that $E_\emptyset =E_{t,\bar{s}}$.
For any $J\in\JJ$, define
\begin{equation}\label{eq:dj}
d_{J}=\prod_{\overline{s}<k<t,k\not\in J}c_{k},
\end{equation}
where it is understood that $d_{\{\overline{s-1},\overline{s-2},\ldots,\overline{1},1,2,\ldots, t-1\}}=1$.
Introduce the following vector in $M(\la)$:
\begin{equation}
  \label{sing}
S_{-\beta}=\sum_{J\in\JJ}d_J E_J, %E_{t,j_p}E_{j_p,j_{p-1}}\cdots E_{j_1,\overline{s}}.
\end{equation}
which contains a unique term (called the {\em leading term}) with $J$ of maximal cardinality $E_{t,t-1}\cdots E_{2,1}E_{1,\overline{1}}E_{\overline{1},\overline{2}}\cdots E_{\overline{s-1},\overline{s}}$ of coefficient $1$.

\begin{thm}\label{thm:Vsingular}
Let $\beta=\delta_s-\epsilon_t$ and let $\lambda=(a_m,a_{m-1},\ldots,a_1~|~b_1,b_2,\ldots,b_n)$ be a $\beta$-atypical weight.
Then the element $S_{-\beta}\hwv$ is the unique (up to a nonzero scalar multiple) singular vector in $M(\lambda)$ of weight $\lambda-\beta$.
In particular, we have
\begin{equation*}
\mathrm{Hom}_\mathfrak{g} \big(M(\lambda-\beta),M(\lambda) \big)=\mathbb{C}.
\end{equation*}
\end{thm}

\begin{proof}
To show that $S_{-\beta}\hwv$ is singular, it suffices to verify that
\[
E_{\overline{i},\overline{i-1}}\cdot S_{-\beta}\hwv
= E_{j-1,j}\cdot S_{-\beta}\hwv
=E_{\overline{1},1}\cdot S_{-\beta}\hwv =0,
\quad (\forall 1< i\leq m, 1< j\leq n).
\]

Indeed, it is trivial that
$%\begin{align*}
E_{j-1,j}\cdot S_{-\beta}\hwv=0 \quad (\forall j> t).
$ %\end{align*}
Furthermore,
\begin{align*}
E_{t-1,t}\cdot S_{-\beta}\hwv
&=E_{t-1,t}\cdot\sum_{t-1\in J\in\JJ} \Big(d_J( E_{t,t-1}E_{t-1,\sharp}\cdots)+d_Jc_{t-1}(E_{t,\sharp}\cdots) \Big)\\
&=\sum_{t-1\in J\in\JJ}d_J(b_{t-1}-b_t+1+c_{t-1})(E_{t-1,\sharp}\cdots)=0,
\end{align*}
where the last equation uses the identity $b_{t-1}-b_t+1+c_{t-1}=0$ which follows by \eqref{cvalue2}.

For $1< j<t$, we have
\begin{align*}
&E_{j-1,j}\cdot S_{-\beta}\hwv\\
&=E_{j-1,j}\cdot\sum_{j,j-1\in J\in\JJ}\left(
\begin{array}{l}
d_J(\cdots E_{*,j}E_{j,j-1}E_{j-1,\sharp}\cdots)+d_Jc_j(\cdots E_{*,j-1}E_{j-1,\sharp}\cdots)\\
+d_Jc_{j-1}(\cdots E_{*,j}E_{j,\sharp}\cdots)
+d_Jc_jc_{j-1}(\cdots E_{*,\sharp}\cdots)
\end{array}
\right)\hwv\\
&=\sum_{j,j-1\in J\in\JJ}d_J(b_{j-1}-b_j+1-c_j+c_{j-1})(\cdots E_{*,j}E_{j-1,\sharp}\cdots)\hwv=0,
\end{align*}
where the last equation uses the identity $b_{j-1}-b_j+1-c_j+c_{j-1}=0$ which follows by \eqref{cvalue2}.

Summarizing we have known $E_{j-1,j}\cdot S_{-\beta}\hwv=0 ~(\forall 1<j\leq n)$. In an entirely similar way, we can show $E_{\overline{i},\overline{i-1}}\cdot S_{-\beta}\hwv=0 ~(\forall 1<i\leq m)$.

In addition, we compute that
\begin{align*}
&E_{\overline{1},1}\cdot S_{-\beta}\hwv\\
&=E_{\overline{1},1}\cdot\sum_{\overline{1},1\in J\in\JJ}\left(
\begin{array}{l}
d_J(\cdots E_{*,1}E_{1,\overline{1}}E_{\overline{1},\sharp}\cdots)+d_Jc_{\overline{1}}(\cdots E_{*,1}E_{1,\sharp}\cdots)\\
+d_Jc_{1}(\cdots E_{*,\overline{1}}E_{\overline{1},\sharp}\cdots)
+d_Jc_{\overline{1}}c_{1}(\cdots E_{*,\sharp}\cdots)
\end{array}
\right)\hwv\\
&=\sum_{\overline{1},1\in J\in\JJ}d_J(b_{1}+a_1+1+c_1+c_{\overline{1}})(\cdots E_{*,1}E_{\overline{1},\sharp}\cdots)\hwv\\
&=\sum_{\overline{1},1\in J\in\JJ}d_J(a_s+b_t+s-t)(\cdots E_{*,1}E_{\overline{1},\sharp}\cdots)\hwv \quad\quad\quad
\mbox{by \eqref{cvalue1}-\eqref{cvalue2}}\\
&=0 \quad\quad\quad \mbox{by \eqref{asbt}}.
\end{align*}
(In the above calculation, we implicitly assume $s,t>1$. If either $s=1$ or $t=1$, the argument is similar and much easier.)

Hence we have verified that $S_{-\beta}\hwv$ is a singular vector of weight ($\lambda-\beta$) in $M(\lambda)$.

It remains to prove the uniqueness. In a nutsell, the reason for the uniqueness is that the coefficients of summands in $S_{-\beta}\hwv$ (see \eqref{sing}) are determined recursively by the requirement that $S_{-\beta}\hwv$ is singular once the leading coefficient is fixed. More precisely, let
\[
S'_{-\beta}=\sum_{J\in\JJ}z_JE_{t,j_p}E_{j_p,j_{p-1}}\cdots E_{j_1,\overline{s}}, \qquad (\text{for } z_J\in\mathbb{C})
\]
be a vector of weight $\la-\beta$ which is annihilated by $\mathfrak{n}^+$. Subtracting $S'_{-\beta}$ by a suitable multiple of $S_{-\beta}$ if necessary, we may assume
\begin{equation}\label{za0}
z_{\{\overline{s-1},\overline{s-2},\ldots,\overline{1},1,2,\ldots, t-1\}}=0.
\end{equation}
We will show that $z_J=0$ for any $J\in\JJ$ by downward induction on the cardinality $|J|$.

The base step is \eqref{za0}, as $J=\{\overline{s-1},\ldots,\overline{1},1,\ldots,t-1\}$ has the maximal cardinality among $J\in \JJ$ in \eqref{eq:JJ}. Suppose we have known that $z_{J'}=0$ for any $J'\in\JJ$ with $|J'|>k$.
We consider $J\in \JJ$ with $|J|=k$.

If $t-1\not\in J$, it follows from $E_{t-1,t}\cdot S'_{-\beta}\hwv=0$ that
$z_{J\cup\{t-1\}}(b_{t-1}-b_t+1)+z_J=0.$
Thus $z_J=0$ because of $z_{J\cup\{t-1\}}=0$ by the inductive assumption.

If $t-1\in J$ but $t-2\not\in J$, it follows from $E_{t-2,t-1}\cdot S'_{-\beta}\hwv=0$ that $$z_{J\cup\{t-2\}}(b_{t-2}-b_{t-1}+1)+z_J-z_{(J\cup\{t-2\})\setminus\{t-1\}}=0.$$
Thanks to $z_{J\cup\{t-2\}}=0$ by inductive assumption and $z_{(J\cup\{t-2\})\setminus\{t-1\}}=0$ established in the previous case, we obtain $z_J=0$.

Similarly, we can show $z_J=0$ if $t-1,t-2\in J$ but $t-3\not\in J$, by $E_{t-3,t-2}\cdot S'_{-\beta}\hwv=0$. Repeating this procedure, we see that $z_J=0$ for all $J\in\JJ$ with $|J|=k$.

Therefore we obtain that $S'_{-\beta}=0$, and the uniqueness follows.
The uniqueness of the singular vector implies that (and is equivalent to) $\mathrm{Hom}_\mathfrak{g} \big(M(\lambda-\beta), M(\lambda) \big)=\mathbb{C}$.  %since the image of the highest weight vector of $M(\lambda-\beta)$ under a nonzero homomophism in this Hom space is a singular vector of $M(\la)$.
%The theorem is proved.
\end{proof}

%\begin{rem}
It was known \cite[Lemma 5.4]{Se96} that $\mathrm{Hom}_\mathfrak{g} \big(M(\lambda-\beta),M(\lambda) \big) \neq 0$.
%\end{rem}

\begin{example}
Let us write down explicitly some cases of odd singular vectors in Theorem~\ref{thm:Vsingular}. Keep the notation
$\lambda=(a_m,\ldots,a_1~|~b_1,\ldots,b_n)\in\mathfrak{h}^*$.
\begin{enumerate}
\item
%Take $\mathfrak{g}=\mathfrak{gl}(2|1)$.
Let $\beta=\delta_2-\epsilon_1$ and let $\lambda$ be $\beta$-atypical, i.e., $a_2+b_1+1=0$.
Then
\[
S_{-\beta}=E_{1, \overline{1}}E_{\overline{1}, \overline{2}}+(a_2-a_1)E_{1,\overline{2}}.
\]

\item
%Take $\mathfrak{g}=\mathfrak{gl}(3|1)$.
Let $\beta=\delta_3-\varepsilon_1$ and let $\lambda$ be $\beta$-atypical, i.e., $a_3+b_1+2=0$.
Then
\begin{align*}
S_{-\beta}=&
E_{1,\overline{1}}E_{\overline{1},\overline{2}}E_{\overline{2},\overline{3}}+
(a_3-a_2)E_{1,\overline{1}}E_{\overline{1},\overline{3}}\\
&+(a_3-a_1+1)E_{1,\overline{2}}E_{\overline{2},\overline{3}}+
(a_3-a_1+1)(a_3-a_2)E_{1,\overline{3}}.
\end{align*}

\item
%Take $\mathfrak{g}=\mathfrak{gl}(2|2)$.
Let $\beta=\delta_2-\varepsilon_2$ and let $\lambda$ be $\beta$-atypical, i.e., $a_2+b_2=0$.
Then
\begin{align*}
S_{-\beta}=&E_{2,1}E_{1,\overline{1}}E_{\overline{1},\overline{2}}
+(a_2-a_1)E_{2,1}E_{1,\overline{2}}\\
&+(b_2-b_1-1)E_{2,\overline{1}}E_{\overline{1},\overline{2}}
+(a_2-a_1)(b_2-b_1-1)E_{2,\overline{2}}.
\end{align*}

%\item
%%Take $\mathfrak{g}=\mathfrak{gl}(3|3)$.
%Let $\beta=\delta_3-\varepsilon_3$ and let $\lambda=(a_3,a_2,a_1~|~b_1,b_2,b_3)$ be $\beta$-atypical (i.e. $a_3+b_3=0$). Then
%\begin{align*}
%&S_{-\beta}=\\
%&E_{3,2}E_{2,1}E_{1,\overline{1}}E_{\overline{1},\overline{2}}E_{\overline{2},\overline{3}}
%+(a_3-a_2)E_{3,2}E_{2,1}E_{1,\overline{1}}E_{\overline{1},\overline{3}}
%+(a_3-a_1+1)E_{3,2}E_{2,1}E_{1,\overline{2}}E_{\overline{2},\overline{3}}\\
%&+(b_3-b_1-2)E_{3,2}E_{2,\overline{1}}E_{\overline{1},\overline{2}}E_{\overline{2},\overline{3}}
%+(b_3-b_2-1)E_{3,1}E_{1,\overline{1}}E_{\overline{1},\overline{2}}E_{\overline{2},\overline{3}}\\
%&+(a_3-a_1+1)(a_3-a_2)E_{3,2}E_{2,1}E_{1,\overline{3}}
%+(a_3-a_1+1)(b_3-b_1-2)E_{3,2}E_{2,\overline{2}}E_{\overline{2},\overline{3}}\\
%&+(b_3-b_1-2)(b_3-b_2-1)E_{3,\overline{1}}E_{\overline{1},\overline{2}}E_{\overline{2},\overline{3}}
%+(a_3-a_1+1)(a_3-a_2)(b_3-b_1-2)E_{3,2}E_{2,\overline{3}}\\
%&+(a_3-a_1+1)(b_3-b_1-2)(b_3-b_2-1)E_{3,\overline{2}}E_{\overline{2},\overline{3}}
%+(a_3-a_2)(b_3-b_1-2)E_{3,2}E_{2,\overline{1}}E_{\overline{1},\overline{3}}\\
%&+(a_3-a_2)(b_3-b_2-1)E_{3,1}E_{1,\overline{1}}E_{\overline{1},\overline{3}}
%+(a_3-a_1+1)(b_3-b_2-1)E_{3,1}E_{1,\overline{2}}E_{\overline{2},\overline{3}}\\
%&+(a_3-a_2)(b_3-b_1-2)(b_3-b_2-1)E_{3,\overline{1}}E_{\overline{1},\overline{3}}
%+(a_3-a_1+1)(a_3-a_2)(b_3-b_2-1)E_{3,1}E_{1,\overline{3}}\\
%&+(a_3-a_1+1)(a_3-a_2)(b_3-b_1-2)(b_3-b_2-1)E_{3,\overline{3}}.
%\end{align*}
\end{enumerate}
\end{example}

Inspired by the formulas \eqref{cvalue1}-\eqref{cvalue2}, we introduce the following elements in $\mathbb C +\mathfrak h \subset U(\mathfrak h)$:
\begin{align*}
%\label{C1}
C_{\overline{s-1}}=E_{\overline{s},\overline{s}}-E_{\overline{s-1},\overline{s-1}},
\quad
& C_{\overline{s-2}}=E_{\overline{s},\overline{s}}-E_{\overline{s-2},\overline{s-2}}+1,
\quad \ldots,
%\quad
&C_{\overline{1}}=E_{\overline{s},\overline{s}}-E_{\overline{1},\overline{1}}+s-2,
\\
C_{t-1}=E_{t,t}-E_{t-1,t-1}-1,
\quad
& C_{t-2}=E_{t,t}-E_{t-2,t-2}-2,
\quad
 \ldots,
%\quad
&C_1=E_{t,t}-E_{1,1}-t+1.
\end{align*}
Denote the Borel subalgebra $\mathfrak{b}^- =\mathfrak{n}^- \oplus \mathfrak{h}$. By definition, a Shapovalov element  is an element $\theta_{\beta} \in \mathrm{U}(\mathfrak{b}^-)_{-\beta}$ such that $\theta_{\beta} \hwv$ is a singular vector of weight $\la -\beta$ for any $\beta$-atypical weight $\la$. We have the following corollary to Theorem~\ref{thm:Vsingular}.

\begin{cor}
 \label{cor:Selement}
Let $\beta=\delta_s-\epsilon_t \in \Phi_1^+$. Then the element
\begin{equation*}
\theta_{\beta} =\sum_{J\in\JJ} E_J \Big(\prod_{\overline{s}<k<t,k\not\in J}C_{k} \Big)
\end{equation*}
is the unique element in $\mathrm{U}(\mathfrak{b}^-)_{-\beta}$ which satisfies $\theta_\beta \hwv =S_{-\beta}\hwv$ for any $\beta$-atypical weight $\la$; hence it is a Shapovalov element.
\end{cor}
Some other more complicated expressions for a Shapovalov element associated to $\beta\in \Phi_1^+$ in terms of non-commutative determinants were also obtained by Musson in \cite[Theorems~9.1, 9.2]{Mu17}. By the uniqueness, these two versions of odd Shapovalov elements coincide.

%%%
\subsection{Singular vectors in Kac modules}

For any dominant integral  weight $\lambda\in P^+$, let $L^{0}(\lambda)$ be the finite dimensional irreducible
$\mathfrak{g}_0$-module of highest weight $\lambda$. Extending the $\mathfrak g_{0}$-module $L^0(\la)$ to a $\mathfrak{g}_0\oplus \mathfrak{g}_1$-module
by a trivial $\mathfrak{g}_1$-action, we define the (finite-dimensional) Kac module as
\begin{equation}\label{kacmodule}
K(\lambda):=\mathrm{Ind}_{\mathfrak{g}_0\oplus \mathfrak{g}_1}^{\mathfrak{g}}L^{0}(\lambda).
%\cong \mathrm{U}(\mathfrak{g})\otimes_{\mathrm{U}(\mathfrak{g}_0\oplus \mathfrak{g}_1)}L^{0}(\lambda)
\end{equation}

For any dominant integral weight $\lambda=(a_m,a_{m-1},\ldots,a_1~|~b_1,b_2,\ldots,b_n)\in P^+,$
it is well known that
\begin{equation}\label{KMI}
K(\lambda)=M(\lambda)/I_\lambda
\end{equation} where $I_\lambda$
is the $\mathfrak{g}$-submodule generated by the singular vectors
\begin{equation}\label{generatorsI}
\left\{ E_{\overline{i-1},\overline{i}}^{a_i-a_{i-1}+1}\hwv,  E_{j+1,j}^{b_{j}-b_{j+1}+1} \hwv~\middle|~1<i\leq m, 1\leq j<n\right\}.
\end{equation}

Let $J_\lambda$ be the $\mathfrak{n}^-$-submodule of $M(\lambda)$ generated by
\begin{equation}
 \label{eq:ideal J}
\left\{E_{\overline{i-1},\overline{i}} \hwv, E_{j+1,j} \hwv~\middle|~1<i\leq m, 1\leq j<n\right\}.
\end{equation}

\begin{lemma}\label{EijinJ}
\begin{enumerate}
\item
For $n\geq i>j\geq1 ~\mbox{or}~ \overline{1}\geq i>j\geq \overline{m}$, we have  $E_{i,j}\hwv \in J_\lambda$.
\item
For $n\geq t\geq 1> \overline{s}\geq \overline{m}$, we have $E_{t,\overline{s}} \hwv \not \in J_{\lambda}$.
\end{enumerate}
\end{lemma}

\begin{proof}
(1) For any $n\geq i>j\geq1$, we have
\begin{align*}
E_{i,j}\hwv&=E_{i,j+1}E_{j+1,j}\hwv-E_{j+1,j}E_{i,j+1}\hwv\\
&\equiv
-E_{j+1,j}E_{i,j+1}\hwv\equiv\cdots \\
&\equiv (-1)^{i-j-1}E_{j+1,j}E_{j+2,j+1}\cdots E_{i,i-1}\hwv
\equiv 0 \quad (\mathrm{mod}~ J_\lambda).
\end{align*}
The remaining case (with $\overline{1}\geq i>j\geq \overline{m}$) follows by an entirely similar argument.
%\end{proof}

%\begin{lemma} %\label{Ets}
%For any $n\geq t\geq 1> \overline{s}\geq \overline{m}$, the element $E_{t,\overline{s}} \hwv$ does not appear in $J_{\lambda}$.
%\end{lemma}

%\begin{proof}
(2) Let $\mathfrak{n}_{\bar 0}^- =\mathfrak{n}^- \cap \mathfrak{g}_{\bar{0}}, \mathfrak{n}_{\bar 1}^+ =\mathfrak{n}^+ \cap \mathfrak{g}_{\bar{1}}$, and so that $\mathfrak{n}^-=\mathfrak{n}_{\bar 0}^-\oplus \mathfrak{n}_{\bar 1}^-$. Thanks to the PBW theorem, we have
$%\[
\mathrm{U}(\mathfrak{n}^-)=\mathrm{U}(\mathfrak{n}_{\bar 1}^-)\mathrm{U}(\mathfrak{n}_{\bar 0}^-)
=\mathrm{U}(\mathfrak{n}_{\bar 1}^-)\oplus\mathrm{U}(\mathfrak{n}^-)\mathfrak{n}_{\bar 0}^-.
$ %\]
We have $E_{t,\overline{s}} \in \mathrm{U}(\mathfrak{n}_1^-)$, and hence $E_{t,\overline{s}}\not\in \mathrm{U}(\mathfrak{n}^-)\mathfrak{n}_{\bar 0}^-$. Therefore we have $E_{t,\overline{s}}\hwv \not\in\mathrm{U}(\mathfrak{n}^-)
\mathfrak{n}_{\bar 0}^-\hwv=J_\lambda.$
\end{proof}

We continue to denote by $\hwv$ the highest weight vector in $K(\la)$.

\begin{theorem}
  \label{thm:Ksingular}
Let $\beta\in\Phi_1^+$, $\lambda \in P^+$ and $\lambda-\beta\in P^+$ be such that $(\lambda+\rho,\beta)=0$.
%Then the singular vector $S_{-\beta}\hwv\not\in J_\lambda$ and hence $S_{-\beta}\hwv\not\in I_\lambda$.
Then $S_{-\beta}\hwv$ is a singular vector of weight $\la -\beta$, unique up to a scalar multiple, in the Kac module $K(\lambda)$. In particular, we have
\begin{equation} \label{eq:SK}
\mathrm{Hom}_\mathfrak{g} \big(K(\lambda-\beta),K(\lambda) \big)=\mathbb{C}.
\end{equation}
\end{theorem}

\begin{proof}
Recall $K(\la) =M(\la)/I_\la$ in \eqref{KMI}. We check that $S_{-\beta}\hwv$ (regarded as a vector in $M(\la)$) does not lie in the submodule $I_\la$.
To that end, let
\[
\mathcal{B}=\{J=\{j_1,j_2,\ldots, j_p\}\subset I~|~ 1\leq j_1<j_2<\cdots<j_p<t, \, 0\le p \le t-1\}\subset \JJ,
\]
where it is understood that $J=\emptyset$ if $p=0$. Moreover, below it is understood that $E_{t,j_p}E_{j_p-j_{p-1}}\cdots E_{j_1,\overline{s}}=E_{t,\overline{s}}$ if $p=0$.
By applying Lemma \ref{EijinJ}(1) repeatedly we have
\begin{align}
S_{-\beta}\hwv & \equiv \sum_{J\in\mathcal{B}}d_JE_{t,j_p}E_{j_p-j_{p-1}}\cdots E_{j_1,\overline{s}}\hwv
\notag \\
&\equiv \Big(\sum_{J\in\mathcal{B}}d_J \Big)E_{t,\overline{s}}\hwv
 \notag \\
&= \Big(\prod_{\overline{s}<i\leq \overline{1}}c_i\Big)
\Big(\prod_{1\leq j<t}(1+c_j)\Big) E_{t,\overline{s}}\hwv \quad (\mathrm{mod}\ J_\lambda).
 \label{equiv}
\end{align}
We know $E_{t,\overline{s}}\hwv\not\in J_\lambda$ by Lemma \ref{EijinJ}(2). Note that $\lambda-\beta\in P^+$ implies $c_i>0$ for $\overline{s}<i\leq \overline{1}$ and $1+c_{j}<0$ for $1\leq j<t$. It follows by \eqref{equiv} that $S_{-\beta}\hwv\not\in J_\lambda$.

Since the $\mathfrak{n}^-$-submodule generated by the singular vectors in \eqref{generatorsI} is equal to $I_\lambda$, we have
$I_\lambda\subseteq J_\lambda$ by the definition \eqref{eq:ideal J} of $J_\la$, and hence $S_{-\beta}\hwv\not\in I_\lambda$.
We conclude by Theorem~\ref{thm:Vsingular} that $S_{-\beta}\hwv$ descends to a singular vector in the Kac module $K(\lambda)$.

To complete the proof of the theorem, it remains to prove that
\begin{equation*}
\dim \mathrm{Hom}_\mathfrak{g} \big(K(\lambda-\beta),K(\lambda) \big) \leq 1,
\end{equation*}
or equivalently, by Frobenius reciprocity,
\begin{equation}  \label{eq:1}
\dim \mathrm{Hom}_{\mathfrak{g}_{\bar 0}} \big(L^0(\lambda-\beta), \Lambda(\mathfrak g_{-1}) \otimes L^0(\lambda) \big) \leq 1.
\end{equation}
The statement \eqref{eq:1} was stated by Serganova in \cite[Proof of Theorem 5.5]{Se96} who skipped the proof. Let us provide a short proof below (which is a simplified version of our original argument thanks to inputs from Shun-Jen Cheng). A singular vector for the $\mathfrak{g}_{\bar 0}$-module $\Lambda(\mathfrak g_{-1}) \otimes L^0(\lambda)$ must be of the form $x \otimes \hwv + \cdots \otimes U(\mathfrak n_{\bar 0}^-) \mathfrak n^-_{\bar 0} \hwv$ for some $0\neq x\in \Lambda(\mathfrak g_{-1})$, and so is of weight $\la +\mu$ for some weight $\mu$ for $\Lambda(\mathfrak g_{-1})$. Now observe that the $(-\beta)$-weight subspace $\Lambda(\mathfrak g_{-1})_{-\beta} =\Lambda^1(\mathfrak g_{-1})_{-\beta}$ is one-dimensional; hence \eqref{eq:1} follows.

The proof of the theorem is completed.
\end{proof}

\begin{rem}
The equality \eqref{eq:SK} was proved by Serganova in \cite[Theorem 5.5]{Se96} in a different approach, which played a fundamental role in the Kazhdan-Lusztig theory for the finite-dimensional module category of $\mathfrak{gl}(m|n)$. Our proof that $\mathrm{Hom}_\mathfrak{g} \big(K(\lambda-\beta),K(\lambda) \big) \neq 0$ is constructive and direct.
\end{rem}

\begin{rem}
The assumption $\la -\beta \in P^+$ in Theorem~\ref{thm:Ksingular} cannot be removed. Consider $\mathfrak g=\mathfrak{gl}(2|1)$ and $\beta=\delta_2-\varepsilon_1 \in \Phi_1^+$. A weight $\lambda= (a_2, a_1\mid b_1)$ with $b_1= -a_2 -1$ is $\beta$-atypical. The formula \eqref{sing} reads $S_{-\beta}=E_{1,\overline{1}}E_{\overline{1},\overline{2}}+c_{\overline{1}}E_{1,\overline{2}}.$ If $a_1=a_2$, then by \eqref{cvalue1} $c_{\overline{1}}=0$ and hence $S_{-\beta} \hwv=0$ in $K(\la)$.
\end{rem}

%%%
\subsection{Other basic types}

Verma $\mathfrak g$-modules $M(\la)$ can be defined as usual with respect to a triangular decomposition $\mathfrak g= \mathfrak n^- \oplus \mathfrak h \oplus \mathfrak n^+$ for any basic Lie superalgebra $\mathfrak g$ (cf. \cite[Chapter~ 1]{CW12} for basic Lie superalgebras with standard choices of positive root systems). %As usual the positive roots are divided into even and odd: $\Phi^+ =\Phi_0^+ \cup \Phi_1^+$.

\begin{conj}\label{conjecture}
Let $\mathfrak g$ be a basic Lie superalgebra.  Assume a weight $\la$ is $\beta$-atypical for an isotropic positive odd root $\beta$. Then \begin{equation*}
\mathrm{Hom}_\mathfrak{g} \big(M(\lambda-\beta),M(\lambda) \big)=\mathbb{C}.
\end{equation*}
\end{conj}
We note that via the study of Shapovalov elements Musson \cite{Mu17} has already established the Hom space in \eqref{eq:1d} is non-vanishing. The conjecture holds for $\mathfrak g= \mathfrak{gl}(m|n)$ by Theorem~\ref{thm:Vsingular}. It also holds for the exceptional Lie superalgebra $D(2|1; \zeta)$ by \cite[Lemma~ 2.3]{CW17}, where formulas for the singular vectors are given.

Below we focus on the ortho-symplectic Lie superalgebras $\mathfrak{osp}$.
A standard set of simple roots for $\mathfrak{g}=\mathfrak{osp}(2m|2n)$ is
\[
\Pi=\{\delta_{i+1}-\delta_{i},\delta_{1}-\epsilon_1,\epsilon_j-\epsilon_{j+1}, \epsilon_{m-1}+\epsilon_m~|~1\leq i<n, 1\leq j<m\},
\]
and the sets of positive roots and roots are
\begin{align*}
\Phi^+ &=\{\delta_i\pm\delta_j,\delta_i\pm\epsilon_k,\epsilon_k\pm\epsilon_l,2\delta_i~|~i>j,k<l\}
\\
\Phi &=\Phi^+\cup(-\Phi^+).
\end{align*}
%The bilinear form $(,)$ on $\mathfrak{h}^*$ satisfies $$(\delta_i,\delta_j)=-\delta_{ij},\quad (\delta_i,\epsilon_j)=(\epsilon_j,\delta_i)=0,\quad (\epsilon_i,\epsilon_j)=\delta_{ij}.$$
%A weight $\lambda\in\mathfrak{h}^*$ is written in terms of the $\delta\epsilon$-basis as $$\lambda=(a_1,a_2,\ldots,a_n~|~b_1,b_2,\ldots,b_m)=\sum_{i=1}^na_i\delta_i+\sum_{j=1}^m b_j\epsilon_j.$$
We have $\rho= (n-m,\ldots,2-m,1-m~|~m-1,m-2,\ldots,0).$
Moreover, the set of positive odd roots is
$$\Phi^+_1=\{\delta_i\pm\epsilon_k~|~1\leq i\leq n, 1\leq k\leq m\}.$$
For $\beta=\delta_s\pm\epsilon_t \in \Phi_1^+$, a weight $\lambda=\sum_{i=1}^n a_i\delta_i+\sum_{k=1}^m b_k\epsilon_k=(a_n,\ldots,a_1~|~b_1,\ldots,b_m)$
is $\beta$-atypical (that is, $(\lambda+\rho,\beta)=0$) if and only if $a_s-m+s=\pm(b_t+m-t)$.

Denote by $e_\alpha$  $(\alpha\in\Phi)$ the Chevalley generators of $\mathfrak{g}$. Regarding $\mathfrak{g}$ as a subalgebra of $\mathfrak{gl}(2m|2n)$, the generators can be chosen explicitly as follows (cf. \cite{CW12}):
\begin{align*}
%\displaybreak
e_{2\delta_i}=E_{\overline{i},\overline{i+n}}, &\quad\quad e_{-2\delta_i}=E_{\overline{i+n},\overline{i}},\\
e_{\delta_i+\delta_j}=E_{\overline{i},\overline{j+n}}+E_{\overline{j},\overline{i+n}}, &\quad\quad e_{-\delta_i-\delta_j}=E_{\overline{j+n},\overline{i}}+E_{\overline{i+n},\overline{j}},  \quad (i\neq j) \\
e_{\delta_i-\delta_j}=E_{\overline{i},\overline{j}}+E_{\overline{j+n},\overline{i+n}}, &\quad\quad e_{\epsilon_k-\epsilon_l}=E_{kl}-E_{l+m,k+m},\\
e_{\epsilon_k+\epsilon_l}=E_{k,l+m}-E_{l,k+m}, &\quad\quad e_{-\epsilon_k-\epsilon_l}=E_{l+m,k}-E_{k+m,l}, \quad(k<l)\\
e_{\delta_i+\epsilon_k}=E_{k,\overline{i+n}}+E_{\overline{i},k+m}, &\quad\quad e_{-\delta_i-\epsilon_k}=E_{k+m,\overline{i}}-E_{\overline{i+n},k},\\
e_{\delta_i-\epsilon_k}=E_{k+m,\overline{i+n}}+E_{\overline{i},k}, &\quad\quad e_{-\delta_i+\epsilon_k}=E_{k,\overline{i}}-E_{\overline{i+n},k+m}.
\end{align*}
(The generators for $\mathfrak{osp}(2m+1|2n)$ can be similarly chosen; cf. \cite{CW12}.)

Recall the set $\mathcal{A}$ in \eqref{eq:JJ}. For any $J=\{j_1,j_2,\ldots, j_p\}\in\mathcal{A}$, we set
$$e_J=e_{\epsilon_t-\epsilon_{j_p}}e_{\epsilon_{j_p}-\epsilon_{j_{p-1}}}\cdots e_{\epsilon_{j_1}-\epsilon_{\overline{s}}},$$
where we use the convention $\epsilon_{\overline{i}}:=\delta_i$. It is understood $e_\emptyset =e_{\epsilon_t-\delta_{s}}$.
Recall the scalar $d_J$ in \eqref{eq:dj}. Now we set
$S_{-\beta}=\sum_{J\in\JJ}d_J e_J.$

\begin{proposition}
Let $\mathfrak{g}=\mathfrak{osp}(2m|2n)$ or $\mathfrak{osp}(2m+1|2n)$. Let $\beta=\delta_s-\epsilon_t$. If a weight $\lambda=(a_n,\ldots,a_1~|~b_1,\ldots,b_m)$
is $\beta$-atypical, then the element $S_{-\beta}\hwv$ is the unique (up to a nonzero scalar multiple) singular vector in $M(\lambda)$ of weight $\lambda-\beta$.
\end{proposition}

\begin{proof}
Follows by the same proof as for Theorem \ref{thm:Vsingular}.
\end{proof}

In particular, Conjecture \ref{conjecture} holds for odd roots of the form $\beta=\delta_s-\epsilon_t$ of the $\mathfrak{osp}$ Lie superalgebras. This singular vector formula does not easily generalize to the setting of an odd root of $\mathfrak{osp}$ of the form $\delta_s +\epsilon_t$.
We end with an example of $\mathfrak{osp}$ of low rank supporting the conjecture in this case. % by constructing a singular vector explicitly.

\begin{example}
Let $\mathfrak{g}=\mathfrak{osp}(6|2)$, $\rho=(-2~|~2,1,0)$ and $\beta=\delta_1+\epsilon_1$. Assume $\lambda=(a~|~b_1,b_2,b_3)$ with $a=b_1+4$. Then $(\lambda+\rho,\beta)=0$. Define
\allowdisplaybreaks
\begin{align*}
S_{-\beta}=&e_{\epsilon_3-\epsilon_2}e_{-\epsilon_3-\epsilon_2}e_{\epsilon_2-\epsilon_1}^2e_{\epsilon_1-\delta}
-(b_2+b_3+2)e_{\epsilon_3-\epsilon_2}e_{-\epsilon_3-\epsilon_1}e_{\epsilon_2-\epsilon_1}e_{\epsilon_1-\delta} \\
&-(b_2-b_3+2)e_{-\epsilon_3-\epsilon_2}e_{\epsilon_3-\epsilon_2}e_{\epsilon_2-\epsilon_1}e_{\epsilon_1-\delta}
-(2b_1+4)e_{\epsilon_3-\epsilon_2}e_{-\epsilon_3-\epsilon_2}e_{\epsilon_2-\epsilon_1}e_{\epsilon_2-\delta}\\
&-[(b_1-b_2)(b_1+b_2+3)+(b_2-b_3+2)(b_2+b_3+2)]e_{-\epsilon_2-\epsilon_1}e_{\epsilon_2-\epsilon_1}e_{\epsilon_1-\delta}\\
&-(b_1-b_2)(b_1+b_2+3)e_{-\epsilon_3-\epsilon_1}e_{-\epsilon_3-\epsilon_1}e_{\epsilon_1-\delta}
+(b_1+b_2+3)(b_1+b_3+2)e_{\epsilon_3-\epsilon_2}e_{\epsilon_2-\epsilon_1}e_{-\epsilon_3-\delta}\\
&+(b_1+b_2+3)(b_1-b_3+2)e_{-\epsilon_3-\epsilon_2}e_{\epsilon_2-\epsilon_1}e_{\epsilon_3-\delta}\\
&+[(b_1-b_2+1)(b_2+b_3+2)-(b_1-b_2)(b_1+b_2+3)]e_{\epsilon_3-\epsilon_2}e_{-\epsilon_3-\epsilon_1}e_{\epsilon_2-\delta}\\
&+[(b_1-b_2+1)(b_2-b_3+2)-(b_1-b_2)(b_1+b_2+3)]e_{-\epsilon_3-\epsilon_2}e_{\epsilon_3-\epsilon_1}e_{\epsilon_2-\delta}\\
&+[(b_1-b_2)^2(b_1+b_2+3)+(b_1-b_2+1)(b_2-b_3+2)(b_2+b_3+2)]e_{-\epsilon_2-\epsilon_1}e_{\epsilon_2-\delta}\\
&+(b_1+b_2+3)(b_1-b_3+2)(b_1+b_3+2)e_{\epsilon_2-\epsilon_1}e_{-\epsilon_2-\delta}\\
&+(b_1-b_2)(b_1+b_2+3)(b_1+b_3+2)e_{\epsilon_3-\epsilon_1}e_{-\epsilon_3-\delta}\\
&+(b_1-b_2)(b_1+b_2+3)(b_1-b_3+2)e_{-\epsilon_3-\epsilon_1}e_{\epsilon_3-\delta}\\
&+(b_1-b_2+1)(b_1+b_2+3)(b_1-b_3+2)(b_1+b_3+2)e_{-\epsilon_1-\delta}.
\end{align*}
Then by a direct computation $S_{-\beta}\hwv$ is the unique (up to a scalar multiple) singular vector in the Verma $\mathfrak{g}$-module $M(\lambda)$ of weight $\lambda-\beta$. It remains a challenging and interesting problem to find a closed formula for the  singular vectors with respect to odd roots $\beta=\delta_s+\epsilon_t$ for general $\mathfrak{osp}$ Lie superalgebras.
\end{example}

%%%%%%%%%
%%%%%%%%%

\end{document}